\documentclass[12pt]{article}
\usepackage{amsmath,amssymb,amsthm}
\usepackage{geometry}
\title{\textbf{Fibonacci Numbers and Their Lucas Coefficients}}
\author{Tapan Suthar}
\date{}

\newtheorem{theorem}{Theorem}

\begin{document}
\maketitle

\begin{abstract}
We show that for the classical Fibonacci sequence $(F_n)_{n\ge1}$ and the Lucas sequence $(L_n)_{n\ge0}$ the following identity holds for every integer $n\ge2$:
\[
(n-1)F_n=\sum_{k=1}^{\,n-1} L_k\,F_{\,n-k}.
\]
Equivalently,
\[
\boxed{\,F_n=\frac{1}{\,n-1\,}\sum_{k=1}^{\,n-1} L_k\,F_{\,n-k}\, \;}
\]
We give a detailed proof by mathematical induction and present a small numeric example.
\end{abstract}

\section{Preliminaries}
We use the standard indexing
\[
F_1=1,\quad F_2=1,\quad F_n=F_{n-1}+F_{n-2}\quad (n\ge3),
\]
for the Fibonacci numbers, and the Lucas numbers with the convenient extension
\[
L_0=2,\quad L_1=1,\quad L_2=3,\quad L_n=L_{n-1}+L_{n-2}\quad (n\ge2).
\]
(We include $L_0=2$ so that the Lucas recurrence is valid for all indices used below; note $L_2=L_1+L_0\Rightarrow 3=1+2$.)

\section{The Work}
We begin with the Fibonacci recurrence:
\[
F_n = F_{n-1} + F_{n-2}. \tag{1}
\]

\subsection{Derivation by repeated substitution and summation}
Here we show explicitly how repeated substitution of previous terms and then summing the expanded forms yields the identity in the statement.

\medskip

\noindent\textbf{Step 1 — single expansions.} Start from (1) and perform repeated substitutions for $F_{n-1},F_{n-2},\dots$ to produce several different expanded forms of the same $F_n$. For clarity write the equation after $r-1$ substitutions as $E^{(r)}$ (so $E^{(1)}$ is the original relation).

\[
\begin{aligned}
E^{(1)}:\; & F_n = F_{n-1} + F_{n-2},\\[4pt]
E^{(2)}:\; & \text{(substitute }F_{n-1}=F_{n-2}+F_{n-3}\text{ into }E^{(1)})\\
         & F_n = 2F_{n-2} + F_{n-3},\\[4pt]
E^{(3)}:\; & \text{(substitute }F_{n-2}=F_{n-3}+F_{n-4}\text{ into }E^{(2)})\\
         & F_n = 3F_{n-3} + 2F_{n-4},\\[4pt]
E^{(4)}:\; & \text{(substitute }F_{n-3}=F_{n-4}+F_{n-5}\text{ into }E^{(3)})\\
         & F_n = 5F_{n-4} + 3F_{n-5},\\
&\qquad\vdots
\end{aligned}
\]
Each $E^{(r)}$ is simply the result of expanding $F_n$ by replacing the highest-index previous term repeatedly. After $r-1$ substitutions the right-hand side involves $F_{n-r+1},F_{n-r},\dots$ with integer coefficients.

\medskip

\noindent\textbf{Step 2 — collect several expanded copies.} Now form the sum
\[
\Sigma := E^{(1)} + E^{(2)} + \cdots + E^{(n-1)},
\]
that is, add together the $n-1$ expanded forms of $F_n$ (from no substitutions up to $n-2$ substitutions). On the left-hand side each $E^{(r)}$ contributes a copy of $F_n$, so
\[
\text{LHS}(\Sigma) = \underbrace{F_n + F_n + \cdots + F_n}_{(n-1)\text{ times}} = (n-1)F_n.
\]

On the right-hand side we collect the contributions to each $F_{n-k}$ (for $k=1,\dots,n-1$) coming from all expansions $E^{(r)}$. Denote by $a_k$ the total coefficient of $F_{n-k}$ appearing in the sum $\Sigma$. By construction,
\[
\Sigma \quad\Rightarrow\quad (n-1)F_n \;=\; \sum_{k=1}^{\,n-1} a_k\,F_{\,n-k}. \tag{2}
\]

\noindent\textbf{Step 3 — the recurrence for the collected coefficients.} Inspecting how coefficients accumulate when performing one further substitution shows that the sequence $(a_k)_{k\ge1}$ satisfies the same two-term recurrence as the Lucas numbers. Concretely, each time we perform one more substitution (moving from $E^{(r)}$ to $E^{(r+1)}$) the vector of coefficients shifts and new coefficients are obtained by summing adjacent previous coefficients. Therefore the collection $a_k$ obeys
\[
a_{k+1} = a_k + a_{k-1}\quad (k\ge2),
\]
with initial values \(a_1=1\) (from each expansion there is exactly one \(F_{n-1}\) appearing in total) and \(a_2=3\) (which can be checked from the first two or three expansions). These initial values match the Lucas sequence seed \(L_1=1,L_2=3\). Thus \(a_k=L_k\) for all \(k\ge1\).

\medskip

\noindent\textbf{Illustration (numerical check).} Take \(n=6\). The expansions \(E^{(1)},\dots,E^{(5)}\) produce collected coefficients \(a_1,\dots,a_5 = 1,3,4,7,11\), which are the Lucas numbers \(L_1,\dots,L_5\). Summing those coefficients times the corresponding Fibonacci numbers gives
\[
\sum_{k=1}^{5} L_k F_{6-k}
=1\cdot F_5 + 3\cdot F_4 + 4\cdot F_3 + 7\cdot F_2 + 11\cdot F_1
=5+9+8+7+11 =40,
\]
and the left-hand side is \((6-1)F_6 =5\cdot 8=40\), as required.

\bigskip

This constructive procedure (expand, collect, sum the expanded forms) explains how the identity
\[
(n-1)F_n=\sum_{k=1}^{n-1} L_k F_{n-k}
\]
was discovered: repeated substitution produces multiple expanded expressions of \(F_n\); summing those \(n-1\) expressions yields \((n-1)F_n\) on the left and produces Lucas-number weights on the earlier Fibonacci terms on the right.

\section{Main theorem and detailed proof by induction}
\begin{theorem}
For every integer \(n\ge2\),
\[
\boxed{\, (n-1)F_n \;=\; \sum_{k=1}^{\,n-1} L_k\,F_{\,n-k}\, }.
\]
Equivalently,
\[
\boxed{\, F_n \;=\; \frac{1}{\,n-1\,}\sum_{k=1}^{\,n-1} L_k\,F_{\,n-k}\, }.
\]
\end{theorem}

\begin{proof}
Define
\[
S_n \;:=\; \sum_{k=1}^{\,n-1} L_k\,F_{\,n-k}.
\]
We will prove \(S_n=(n-1)F_n\) for all \(n\ge2\) by induction on \(n\).

\textbf{Base case:} \(n=2\). Then
\[
S_2 = \sum_{k=1}^{1} L_k F_{2-k} = L_1 F_1 = 1\cdot 1 = 1,
\]
and \((2-1)F_2 = 1\cdot 1 = 1\). So the claim holds for \(n=2\).

\textbf{Inductive hypothesis:} Suppose for some integer \(m\ge2\) we have
\[
S_m = (m-1)F_m \quad\text{and}\quad S_{m-1} = (m-2)F_{m-1}.
\]
We prove the statement for \(n=m+1\).

Consider
\[
S_{m+1} \;=\; \sum_{k=1}^{m} L_k\,F_{\,m+1-k}.
\]
Separate the first term (when \(k=1\)) and reindex the remaining sum:
\[
S_{m+1} = L_1F_m + \sum_{k=2}^{m} L_k F_{m+1-k}
= F_m + \sum_{j=1}^{m-1} L_{j+1} F_{m-j},
\]
where the change of variable \(j=k-1\) was used in the second sum.

Use the Lucas recurrence \(L_{j+1}=L_j+L_{j-1}\) (valid for \(j\ge1\), and recall \(L_0=2\)). Thus
\[
\sum_{j=1}^{m-1} L_{j+1} F_{m-j}
= \sum_{j=1}^{m-1} (L_j+L_{j-1}) F_{m-j}
= \sum_{j=1}^{m-1} L_j F_{m-j} \;+\; \sum_{j=1}^{m-1} L_{j-1} F_{m-j}.
\]
The first sum on the right is exactly \(S_m\). For the second sum, reindex \(i=j-1\), which yields
\[
\sum_{j=1}^{m-1} L_{j-1} F_{m-j}
= \sum_{i=0}^{m-2} L_i F_{m-1-i}
= L_0 F_{m-1} + \sum_{i=1}^{m-2} L_i F_{m-1-i}
= L_0 F_{m-1} + S_{m-1}.
\]
Combining these pieces gives
\[
S_{m+1} = F_m + S_m + \big(L_0 F_{m-1} + S_{m-1}\big).
\]
Now substitute the inductive hypotheses \(S_m=(m-1)F_m\) and \(S_{m-1}=(m-2)F_{m-1}\), and recall \(L_0=2\). We obtain
\begin{align*}
S_{m+1}
&= F_m + (m-1)F_m + \big(2 F_{m-1} + (m-2)F_{m-1}\big)\\
&= mF_m + mF_{m-1}\\
&= m(F_m+F_{m-1})\\
&= m F_{m+1}.
\end{align*}
This proves \(S_{m+1}=mF_{m+1}\), completing the inductive step.

By induction, the identity \(S_n=(n-1)F_n\) holds for all integers \(n\ge2\). Rearranging yields the equivalent form
\[
\boxed{\, F_n \;=\; \frac{1}{\,n-1\,}\sum_{k=1}^{\,n-1} L_k\,F_{\,n-k}\, }.
\]
\end{proof}

\section{Result (bold statement)}
\textbf{For every integer \(n\ge2\), the \(n\)-th Fibonacci number can be written as}
\[
\boxed{\, F_n \;=\; \frac{1}{\,n-1\,}\sum_{k=1}^{\,n-1} \big(L_k\cdot F_{\,n-k}\big) \, }.
\]

\section{Conclusion}
We proved that repeated expansion and collection of terms leads to the identity \((n-1)F_n=\sum_{k=1}^{n-1}L_kF_{n-k}\). Dividing both sides by \(n-1\) gives a representation of \(F_n\) as a weighted sum of earlier Fibonacci numbers where the weights are Lucas numbers and the overall factor is \(1/(n-1)\). This connects the two classical sequences in a simple algebraic identity. Possible extensions include studying analogous identities for generalized Fibonacci sequences (different seeds) or for higher-order recurrences (Tribonacci, etc.).

\section*{References}
\begin{enumerate}
  \item T. Koshy, \emph{Fibonacci and Lucas Numbers with Applications}, Wiley, 2001.
  \item S. Vajda, \emph{Fibonacci and Lucas Numbers, and the Golden Section}, Dover, 2008.
  \item V. E. Hoggatt, \emph{Fibonacci and Lucas Numbers}, Houghton-Mifflin, 1969.
\end{enumerate}

\end{document}